\documentclass[11pt]{amsart}
\usepackage{graphicx}
\usepackage{graphicx}
\usepackage{amssymb,amsmath}
\usepackage[utf8]{inputenc}
\usepackage[T1]{fontenc}
\newcommand{\qbinom}[2] {{#1 \brack #2}_q}

\newtheorem{thm}{Theorem}[section]
\newtheorem{cor}[thm]{Corollary}
\newtheorem{lem}[thm]{Lemma}
\newtheorem{prop}[thm]{Proposition}
\theoremstyle{definition}
\newtheorem{defn}[thm]{Definition}
\theoremstyle{remark}
\newtheorem{rem}[thm]{Remark}
\newtheorem{exa}[thm]{Example}
\numberwithin{equation}{section}

\newcommand{\be}{\begin{equation}}
\newcommand{\ee}{\end{equation}}

\newcommand{\bea}{\begin{eqnarray}}
\newcommand{\eea}{\end{eqnarray}}

\begin{document}

\title[]{ The $q$-Lidstone series involving $q$-Bernoulli and $q$-Euler polynomials generated by the third Jackson $q$-Bessel function}%
\author{Z. Mansour and M. AL-Towailb}%
\address{Z. Mansour, Department of Mathematics, Faculty of Science, Cairo University, Giza, Egypt.}%
\email{zeinab@sci.cu.edu.eg}%
\address{M. AL-Towailb, Department of Computer Science and Engineering, King Saud University, Riyadh, KSA}%
\email{mtowaileb@ksu.edu.sa}%


\begin{abstract}
In this paper, we present $q$-Bernoulli and $q$-Euler polynomials generated by the third Jackson $q$-Bessel function to construct new types of $q$-Lidstone expansion theorem. We prove that the entire function may be expanded in terms of $q$-Lidstone polynomials which are  $q$-Bernoulli polynomials and the coefficients are the even powers of the $q$-derivative $\frac{\delta_q f(z)}{\delta_q z}$ at $0$ and $1$. The other forms expand the function in $q$-Lidstone polynomials based on $q$-Euler polynomials and the coefficients contain the even and odd powers of the $q$-derivative $\frac{\delta_q f(z)}{\delta_q z}$.
\end{abstract}



\maketitle

\section{{\bf Introduction}}\label{Sec.1}
A Lidstone series provides a generalization of Taylor  series that approximates a given function in a neighborhood of two points instead of one \cite{Lidstone}. Recently, Ismail and Mansour \cite{Ismail and Mansour} introduced a $q$-analog of the Lidstone expansion theorem. They proved that, under certain conditions, an entire function $f(z)$ can be expanded in the form
\begin{equation}\label{q-Lidstone series} f(z)= \sum_{n=0}^{\infty} \Big[ A_n(z) D_{q^{-1}}^{2n}\, f(1)- B_n(z)D_{q^{-1}}^{2n}\, f(0)\Big],\end{equation}
where $A_n(z)$ and $B_n(z)$ are the $q$-Lidstone polynomials defined by
$$A_n(z)=\large{\eta}^1_{q^{-1}} B_n(z) \mbox{ and } B_n(z)=\frac{2^{2n+1}}{[2n+1]_q!}B_{2n+1}(z/2;q).$$
Here $\large{\eta}^y_{q^{-1}}$ denotes the $q$-translation operator defined by
\begin{equation*}{\large{\eta}_{q^{-1}}^y} z^n
=q^{\frac{n(n-1)}{2}}z^n(-y/z  ;q^{-1})_n=y^n(-z/y  ;q)_n,\end{equation*}
and $B_{n}(z;q)$ is the $q$-analog of the Bernoulli polynomials which defined by the
generating function
\begin{equation}
\dfrac{t\,E_q(zt)}{E_q(t/2)e_q(t/2)-1}=\sum_{n=0}^{\infty}B_n(z;q)
\frac{t^n}{[n]_q!},
\end{equation}
where $E_q(z)$ and $e_q(z)$ are the $q$-exponential functions defined by
 $$ E_q(z):= \sum_{j=0}^\infty q^{j(j-1)/2}\,\frac{z^j}{[j]_q!}; \, z\in \mathbb{C} \, \mbox{ and } \,
  e_q(z):= \sum_{j=0}^\infty \frac{z^j}{[j]_q!}; \, |z|< 1.$$
\vskip5mm

 This paper aims to construct the $q$-Lidstone polynomials which are $q$-Bernoulli and $q$-Euler polynomials generated by the third Jackson $q$-Bessel function, and then to derive two formula of $q$-Lidstone expansion theorem. More precisely,
 we will prove that the entire function may be expanded in terms  of $q$-Lidstone polynomials in two different forms. In the first form, the $q$-Lidstone polynomials are $q$-Bernoulli polynomials and the coefficients are the even powers of the $q$-derivative $\frac{\delta_q f(z)}{\delta_q z}$ at $0$ and $1$. The other form expand the function in $q$-Lidstone polynomials based on $q$-Euler polynomials and the coefficients contain the even and odd powers of the $q$-derivative $\frac{\delta_q f(z)}{\delta_q z}$.
 The publications \cite{Mansour and AL-Towaileb,Mansour and AL-Towaileb 2}  are the most affiliated with this work.
 \vskip5mm
This article is organized as follows: in Section 2, we state some definitions and present some background on $q$-analysis which we need in our investigations. In Section 3 and Section 4, we introduce $q$-Bernoulli and $q$-Euler polynomials generated by the third Jackson $q$-Bessel function. Section 5 contains a $q$-Lidstone expansion theorem involving $q$-Bernoulli polynomials while Section 6 contains a $q$-Lidstone series
involving $q$-Euler polynomials.

\section{{\bf Definitions and Preliminary results}}

Throughout this paper, unless otherwise is stated, $q$ is a positive number less than one and we follow the notations and terminology in \cite{AMbook,GR}.

The symmetric $q$-difference operator $\delta_q$ is defined by
$$ \delta_q f(z)= f(q^{\frac{1}{2}}z)-f(q^{\frac{-1}{2}}z),$$
(see \cite{Cardoso011,GR}) and then
\be\label{delta}
\frac{\delta_q f(z)}{\delta_q z}:= \dfrac{f(q^{\frac{1}{2}}z)-f(q^{\frac{-1}{2}}z)}{z(q^{\frac{1}{2}}- q^{\frac{-1}{2}})} \quad z\neq 0.
\ee

We use a third $q$-exponential function  $exp_q(z)$  which has the following series representation
 \be\label{Def EX} exp_q(z)=\sum_{n=0}^{\infty} \frac{q^{\frac{n(n-1)}{4}}}{[n]_q!}z^n;\quad z\in\mathbb{C}.\ee
This function has the property $ \displaystyle \lim_{q\rightarrow 1} exp_{q}(z)= e^z$ for $z\in\mathbb{C}$, and it is an entire function of $z$ of order zero (see \cite{GR}).
\begin{rem}
From the identity $[n]_{1/q}!=q^{\frac{n(1-n)}{2}}[n]_q!$, one can verify that
\be\label{exp_q&exp_{1/q}} exp_{q}(z)= exp_{q^{-1}}(z); \quad z\in\mathbb{C}.\ee
\end{rem}
We consider the domain $\displaystyle \Omega:= \{ z\in \mathbb{C}: \, |1-exp_q(z)|< 1\}$.
\begin{lem}\label{inverse of exp.}
Let $z\in \Omega$. Then \begin{equation}\label{the inverse of exp.}
\dfrac{1}{exp_q(z)}:= 1+ \sum_{n=1}^{\infty} c_n\, z^n\,,
\end{equation}
where
\begin{equation*} c_n= \sum_{k=1}^{n} (-1)^{k}\, \sum_{s_1+s_2+\ldots+s_k=n\atop  s_i>0\,(i=1,\ldots,k) } \dfrac{ q^{ \sum_{i=1}^k s_i(s_i-1)/4 }}{[s_1]_q! [s_2]_q!\ldots [s_{k}]_q!}\,.\end{equation*}
\end{lem}
\begin{proof} Observe that, for $z\in \Omega$ the function $\dfrac{1}{exp_q(z)}$ can be represented as
\begin{equation*}
\dfrac{1}{exp_q(z)} := \Big[\frac{1}{1+(exp_q(z)-1)}\Big] = \sum_{k=0}^{\infty}(-1)^k \Big[ exp_q(z)-1 \Big]^k.\end{equation*}
Using the series expansion \eqref{Def EX} of $exp_q(z)$, we get
\begin{eqnarray*}
\dfrac{1}{exp_q(z)} &=&  \sum_{k=0}^{\infty}(-1)^k \left(\sum_{n=1}^{\infty} q^{n(n-1)/4 }\, \frac{z^n}{[n]_q!}\right)^k \\
&=& 1+\sum_{k=1}^{\infty}(-1)^k \left(\sum_{n=1}^{\infty} q^{n(n-1)/4 }\, \frac{z^{n}}{[n]_q!}\right)^k \\
&=&1+\sum_{k=1}^{\infty}(-1)^k\sum_{n=k}^{\infty} z^{n} \sum_{s_1+s_2+\ldots+s_k=n\atop  s_i>0\,(i=1,\ldots,k) }\dfrac{ q^{ \sum_{i=1}^k s_i(s_i-1)/4 }}{[s_1]_q! [s_2]_q!\ldots [s_{k}]_q!}\,.
\end{eqnarray*}
Put   $\displaystyle  a_n(k)= \sum_{s_1+s_2+\ldots+s_k=n\atop  s_i>0\,(i=1,\ldots,k) } \dfrac{ q^{ \sum_{i=1}^k s_i(s_i-1)/4 }}{[s_1]_q! [s_2]_q!\ldots [s_{k}]_q!}$. Then, the power series of  $\dfrac{1}{exp_q(z)}$ takes the form
\begin{equation*}
\dfrac{1}{exp_q(z)}= 1+ \sum_{n=1}^{\infty} z^n \, \sum_{k=1}^{n} (-1)^{k}a_n(k),
\end{equation*}and then we obtain the desired result.
\end{proof}

The $q$-sine and $q$-cosine, $S_q(z)$ and $C_q(z)$, are defined by $$ exp_q(iz):= C_q(z)+iS_q(z),$$ where
\begin{equation}\label{S&C}\begin{split} C_q(z) &:= \sum_{n=0}^{\infty} (-1)^n\, \frac{q^{n(n-\frac{1}{2})}}{[2n]_q!}z^{2n}, \\ S_q(z) &:= \sum_{n=0}^{\infty}(-1)^n\, \frac{q^{n(n+\frac{1}{2})}}{[2n+1]_q!}z^{2n+1}.\end{split}
 \end{equation}
 These functions can be written in terms of the third Jackson $q$-Bessel function or (Hahn-Exton $q$-Bessel
function \cite{Koelink and Swarttouw}) as
\begin{equation*}
\begin{split}
  C_q(z)&:= q^{-\frac{3}{8}}\, \frac{(q^2;q^2)_{\infty}}{(q;q^2)_{\infty}}((1-q)z)^{\frac{1}{2}}\, J^{(3)}_{-\frac{1}{2}}(q^{\frac{-3}{4}}(1-q)z;q^2),\\ S_q(z)&:= q^{\frac{1}{8}}\, \frac{(q^2;q^2)_{\infty}}{(q;q^2)_{\infty}}((1-q)z)^{\frac{1}{2}}\, J^{(3)}_{\frac{1}{2}}(q^{\frac{-1}{4}}(1-q)z;q^2),
   \end{split}\end{equation*}
and satisfy
\be\label{delta Sine and Cosine}
\frac{\delta_q C_q(wz)}{\delta_q z}= -w\,S_q(wz), \quad \frac{\delta_q S_q(wz)}{\delta_q z}= w\,C_q(wz).\ee
 (see~\cite{Cardoso011,GR}). Therefore, \be\label{delta E} \frac{\delta_q\, exp_q(wz)}{\delta_q z}= w\, exp_q(wz).\ee
\begin{figure}[h]
\includegraphics[width=10 cm, height=5cm]{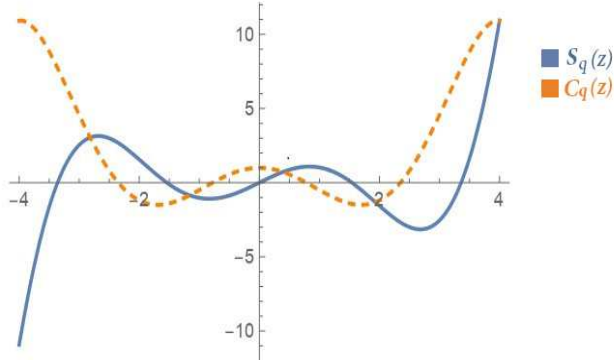}
\caption{The roots of $C_q(z)$ and $S_q(z)$ at $q=\frac{1}{2}$}
\label{Fig(1)}
\end{figure}
Note that since the third Jackson  $q$-Bessel functions have only real roots and the roots are simple (see \cite{Koelink and Swarttouw}), it follows that the roots of $C_q(z)$ and $S_q(z)$ are also real and simple as shown in Figure~\ref{Fig(1)}. Also, because $C_q(z)$ and $S_q(z)$ are respectively even and odd, the roots of these functions are symmetric.

\noindent Throughout this paper we assume that $S_1$ and $C_1$ are the smallest positive zero of the functions $S_q(z)$ and $C_1$, respectively.

Here, the $q$-analog of the hyperbolic functions $\sinh z$ and $\cosh z$ are defined for $z\in \mathbb{C}$ by
\be\label{sh+cosh}
\begin{gathered}
 Sinh_q(z):= -i S_q(iz)=\dfrac{exp_q(z)-exp_q(-z)}{2} \\
 Cosh_q(z):=  C_q(iz) = \dfrac{exp_q(z)+exp_q(-z)}{2}.\end{gathered}\ee
 \vskip 1cm
\section{{\bf A $q$-Bernoulli polynomials generated by the third Jackson $q$-Bessel function }} \label{q-Bernoulli and Euler}
In this section, we  use the third $q$-exponential function  $exp_q(x)$  to define a $q$-analog of the Bernoulli polynomials  which are suitable for our approach.

\begin{defn}\label{q-Bernoulli} A $q$-Bernoulli polynomials $\widetilde{B}_n(z;q)$ are defined by  the generating function
\be\label{GF:Bernoulli-poly} \dfrac{w\,exp_q(zw)\, exp_q(\frac{-w}{2})}{exp_q(\frac{w}{2})-exp_q(\frac{-w}{2})}=\sum_{n=0}^{\infty}\widetilde{B}_n(z;q) \frac{w^n}{[n]_q!},\ee
 and $\widetilde{\beta}_n(q):= \widetilde{B}_n(0;q)$ are  the $q$-Bernoulli numbers. Therefore,
 \be\label{qBernoulli-numbers}
\dfrac{w\, exp_q(-w/2)}{exp_q(\frac{w}{2})-exp_q(\frac{-w}{2})} =\sum_{n=0}^{\infty}\frac{\widetilde{\beta}_n(q)}{[n]_q!}w^n.\ee \end{defn}
\begin{rem}
$\widetilde{B}_{2n+1}(\frac{1}{2};q)=0$. Indeed, for $z=\frac{1}{2}$, the left hand side of Equation \eqref{GF:Bernoulli-poly} is an even function. Therefore, the odd powers of $w$ on the left hand side vanish. Also, note that $$\widetilde{B}_0(z;q)= \dfrac{w\,exp_q(zw)\, exp_q(\frac{-w}{2})}{exp_q(\frac{w}{2})-exp_q(\frac{-w}{2})}\Big|_{w=0}=1.$$
\end{rem}

\begin{prop}\label{prop:1-1}
  The $q$-Bernoulli polynomials $\widetilde{B}_n(z;q)$ are given recursively by $\widetilde{B}_0(z;q)=1$,  and for $n\in\mathbb{N}$ \[\widetilde{B}_n(z;q)=\sum_{k=0}^{n}\qbinom{n}{k} q^{\frac{k(k-1)}{4}}\, \widetilde{\beta}_{n-k}(q) z^k.\] \end{prop}
   \begin{proof}
 By substituting \eqref{qBernoulli-numbers} into \eqref{GF:Bernoulli-poly} and using the series representation of $exp_q(wz)$ we obtain
\begin{equation*}
\begin{split}
    \dfrac{w\,exp_q(zw)\, exp_q(\frac{-w}{2})}{exp_q(\frac{w}{2})-exp_q(\frac{-w}{2})}  &=\sum_{n=0}^{\infty}\widetilde{\beta}_n(q) \frac{w^{n}}{[n]_q!}\sum_{n=0}^{\infty} \frac{q^{\frac{n(n-1)}{4}}}{[n]_q!} (wz)^n \\ &= \sum_{n=0}^{\infty} \frac{w^{n}}{[n]_q!}
 \sum_{k=0}^{n}\qbinom{n}{k} q^{\frac{k(k-1)}{4}}\, \widetilde{\beta}_{n-k}(q) z^k.
 \end{split}
   \end{equation*}
 This implies
   \be \label{Eq:2}
  \sum_{n=0}^{\infty} \frac{w^{n}}{[n]_q!}\,
\sum_{k=0}^{n}\qbinom{n}{k} q^{\frac{k(k-1)}{4}}\, \widetilde{\beta}_{n-k}(q) z^k =\sum_{n=0}^{\infty} \widetilde{B}_n(z;q)\frac{w^n}{[n]_q!}. \ee
Comparing the coefficient of $\frac{w^n}{[n]_q!}$, we obtain the required result.
\end{proof}

\begin{prop}\label{Ber.q and 1/q}
For $n\in\mathbb{N}$ and $z\in \mathbb{C}$, we have
\bea \label{Equiv-Rel:1}\widetilde{B}_n(z;q)&=& q^{\frac{n(n-1)}{2}}\widetilde{B}_n(z;1/q),\\
\label{Equiv-Rel:2}
\widetilde{\beta}_n(q)&=&q^{\frac{n(n-1)}{2}}\widetilde{\beta}_n(1/q).
\eea
\end{prop}
\begin{proof}
By replacing $q$ by $1/q$ on the generating function in \eqref{GF:Bernoulli-poly}, and then using Equation \eqref{exp_q&exp_{1/q}} we obtain
\[\sum_{n=0}^{\infty} q^{\frac{n(n-1)}{2}}\widetilde{B}_n(z;1/q)
\frac{w^n}{[n]_q!}=\sum_{n=0}^{\infty}\widetilde{B}_n(z;q)\frac{w^n}{[n]_q!}.\]
Equating the coefficients of $w^n$ yields \eqref{Equiv-Rel:1} and substituting with $z=0$ in \eqref{Equiv-Rel:1} yields directly \eqref{Equiv-Rel:2}.
\end{proof}
\vskip5mm
\begin{thm}\label{Eq.q-D(TH)}
The $q$-Bernoulli polynomials satisfy the $q$-difference equation
\be\label{Eq.q-D}  \frac{\delta_q \widetilde{B}_n(z;q)}{\delta_q z}=[n]_q\, \widetilde{B}_{n-1}(z;q)\quad
(n\in\mathbb{N}).\ee
\end{thm}
\begin{proof}
Calculating the $q$-derivative $\delta_q$ of the two sides of \eqref{GF:Bernoulli-poly} with respect to the variable $z$ and using Equation \eqref{delta E}, we obtain
\begin{equation*}
\dfrac{w^2 exp_q(zw)exp_q(\frac{-w}{2})}{exp_q(\frac{w}{2})-exp_q(\frac{-w}{2})}=\sum_{n=1}^{\infty}
 \frac{\delta_q \widetilde{B}_n(z;q)}{\delta_q z} \frac{w^{n}}{[n]_q!}.
\end{equation*}
This implies
\begin{equation}\label{11} \sum_{n=1}^{\infty} \frac{\delta_q \widetilde{B}_n(z;q)}{\delta_q z} \,\frac{w^{n}}{[n]_q!}
= \sum_{n=1}^{\infty} \widetilde{B}_{n-1}(z;q)\,\frac{w^{n}}{[n-1]_q!}.\end{equation}
Equating the corresponding $n$th  power of $w$ in the two series of \eqref{11}, we obtain
the required result.
\end{proof}
\begin{cor}
For $ k\geq 2$, we have
\begin{equation*}\
\frac{\delta^{2}_q \widetilde{B}_{k}(z;q)}{\delta_q z^{2}} = [k]_q[k-1]_q \widetilde{B}_{k-2}(z;q).
\end{equation*}
\end{cor}
\begin{proof}
It follows directly by calculating the derivative $\delta_q$ of
\eqref{Eq.q-D} for even and odd index of $\frac{\widetilde{B}_k(z;q)}{[k]_q!}$.
\end{proof}

\begin{prop}
The $q$-Bernoulli numbers of odd index satisfy
\be\label{odd-Bernoulli-numbers} \widetilde{\beta}_1(q)=-\frac{1}{2},
\quad \widetilde{\beta}_{2n+1}(q)=0; \quad n\in\mathbb{N}.\ee
\end{prop}
\begin{proof}
Observe that,
\begin{equation*}
\dfrac{w\, exp_q(\frac{-w}{2})}{exp_q(\frac{w}{2})-exp_q(\frac{-w}{2})} =-w + \dfrac{w\, exp_q(\frac{-w}{2})}{exp_q(\frac{w}{2})-exp_q(\frac{-w}{2})}.\end{equation*}
So, we can write Equation \eqref{qBernoulli-numbers} in the form
\begin{equation*}
\sum_{n=0}^{\infty}\frac{\widetilde{\beta}_n(q)}{[n]_q!} w^n= -w+ \sum_{n=0}^{\infty}\frac{\widetilde{\beta}_n(q)}{[n]_q!} (-w)^n.
\end{equation*}
This implies
\begin{equation*}
\sum_{n=0}^{\infty} (1-(-1)^n)\, \widetilde{\beta}_n(q)\,\frac{w^n}{[n]_q!}= -w.
\end{equation*}
Therefore, $\widetilde{\beta}_1(q)=-\frac{1}{2}$ and $\widetilde{\beta}_{2n+1}(q)=0$ for every
$n\in\mathbb{N}$.
\end{proof}
\vskip5mm
\begin{thm}\label{sum:B}
For $z\in\mathbb{C}$ and $n\in\mathbb{N}$, we have the identity
\begin{equation*}\frac{q^{\frac{n(n-1)}{4}}}{[n]_q!}\,(\frac{-1}{2})^{n} (2q^{\frac{1-n}{2}}\, z;q)_n=\sum_{k=0}^{[\frac{n}{2}]}(\frac{1}{2})^{2k} \frac{q^{\frac{k(2k+1)}{2}}}{[2k+1]_q!}\, \frac{\widetilde{B}_{n-2k}(z;q)}{[n-2k]_q!}.
\end{equation*}
\end{thm}
\begin{proof}
By using \eqref{GF:Bernoulli-poly}, we have
\be\label{Eq.1} w\,exp_q(zw)\, exp_q(\frac{-w}{2})= \Big[exp_q(\frac{w}{2})-exp_q(\frac{-w}{2})\Big]\sum_{n=0}^{\infty}\widetilde{B}_n(z;q) \frac{w^n}{[n]_q!}.\ee
Using the series representation of $exp_q(zw)$, we can prove that
\begin{equation}\label{Eq.2}
exp_q(zw)\, exp_q(\frac{-w}{2})= \sum_{n=0}^{\infty} q^{\frac{n(n-1)}{4}}\frac{w^n}{[n]_q!}(\frac{-1}{2})^{n}(2q^{\frac{1-n}{2}} z;q)_n.  \end{equation}
Substituting \eqref{Eq.2} into \eqref{Eq.1} and using \eqref{Def EX}, we obtain
\begin{equation*} \begin{split}
&\sum_{n=0}^{\infty} q^{\frac{n(n-1)}{4}} \frac{w^n}{[n]_q!}(\frac{-1}{2})^{n}(2q^{\frac{1-n}{2}} z;q)_n \\ &= \sum_{n=0}^{\infty} \frac{q^{\frac{n(2n+1)}{2}}}{[2n+1]_q!} (\frac{w}{2})^{2n}\sum_{n=0}^{\infty}\widetilde{B}_n(z;q) \frac{w^n}{[n]_q!} \\ &= \sum_{n=0}^{\infty} w^n \sum_{k=0}^{[\frac{n}{2}]}(\frac{1}{2})^{2k} \frac{q^{\frac{k(2k+1)}{2}}}{[2k+1]_q!}\, \frac{\widetilde{B}_{n-2k}(z;q)}{[n-2k]_q!}.
\end{split}\end{equation*}
Comparing the coefficient of $w^n$ we obtain the required result.
\end{proof}

\noindent Note that if we substitute with $z=0$ in the identity of Theorem \ref{sum:B}, we get the following recurrence relation:
\begin{cor}
For $n\in\mathbb{N}$, we have
\begin{equation*} \frac{q^{\frac{n(2n-1)}{2}}}{[2n]_q!}\,(\frac{1}{2})^{2n}=\sum_{k=0}^{n}(\frac{1}{2})^{2k} \frac{q^{\frac{k(2k+1)}{2}}}{[2k+1]_q!}\, \frac{\widetilde{\beta}_{2n-2k}(q)}{[2n-2k]_q!}.
\end{equation*}\end{cor}

\noindent As a consequence of the above result, we have \begin{equation*}\begin{gathered} \widetilde{\beta}_{0}(q)=1, \quad \widetilde{\beta}_{1}(q)= -\frac{1}{2}, \quad \widetilde{\beta}_{2}(q)= \dfrac{(1-q^3)q^{\frac{1}{2}}-(1-q)\,q^{\frac{3}{2}}}{4(1-q^3)},\\
\widetilde{\beta}_{3}(q)=0, \quad \widetilde{\beta}_{4}=\dfrac{q^3(q^3;q^2)_2-[3]_q(q^5(1-q)(1-q^3)}{16(1-q^3)^2(1-q^5)}-\\
\dfrac{(1+q^2)(1-q)^2(1-q^5)(q^2(1-q^3)-q^3(1-q)))}{16(1-q^3)^2(1-q^5)}.
 \end{gathered}\end{equation*}
\vskip5mm
\noindent In the following result we prove that the function $Coth_q(z)$ has a $q$-analog of Taylor series expression with only odd exponents for $z$.
\begin{prop}\label{coth_q}
Let $w$ be a complex number such that $0 <|\frac{w}{2}|< C_1$. Then
\begin{equation*}Coth_q(\frac{w}{2})= (\frac{w}{2})^{-1}+ \sum_{n=1}^{\infty}2\widetilde{\beta}_{2n}(q) \frac{w^{2n-1}}{[2n]_q!}.
\end{equation*}\end{prop}
\begin{proof}
 By using Equation \eqref{qBernoulli-numbers} and the identity
\begin{equation*}Coth_q(\frac{w}{2}) = \dfrac{exp_q(\frac{w}{2})+ exp_q(\frac{-w}{2})}{exp_q(\frac{w}{2})-exp_q(\frac{-w}{2})}\, ,\end{equation*}
we obtain \begin{equation*} \begin{split}
\sum_{n=0}^{\infty}\widetilde{\beta}_{n}(q) \frac{w^{n}}{[n]_q!} &= w\,Coth_q(\frac{w}{2})- \dfrac{w\,exp_q(\frac{w}{2})}{exp_q(\frac{w}{2})-exp_q(\frac{-w}{2})}\\
 &= w\,Coth_q(\frac{w}{2})- \sum_{n=0}^{\infty}\widetilde{\beta}_{n}(q) \frac{(-w)^{n}}{[n]_q!}.
\end{split}\end{equation*} Therefore, $\displaystyle w\,Coth_q(\frac{w}{2})= 1+ \sum_{n=1}^{\infty}2\widetilde{\beta}_{2n}(q) \frac{(w)^{2n}}{[2n]_q!}$ and then the result follows.\end{proof}
\vskip5mm
We define the polynomials $\widetilde{A}_n(z;q)$ by the generating function
\be\label{A_n(z;q)}
\dfrac{w\,exp_q(zw)}{exp_q(\frac{w}{2})-exp_q(\frac{-w}{2})}=\sum_{n=0}^{\infty}\widetilde{A}_n(z;q) \frac{w^n}{[n]_q!}.\ee

\begin{prop}\label{Prop:B&A}
For $n\in\mathbb{N}$, the $q$-Bernoulli polynomials $\widetilde{B}_n(z;q)$ can be represented in terms of $\widetilde{A}_{n}(z;q)$ as
\begin{equation}\label{prop:1-2}
\widetilde{B}_n(z;q)=\sum_{k=0}^{n}\qbinom{n}{k} \, (-\frac{1}{2})^k\, q^{\frac{k(k-1)}{4}}\, \widetilde{A}_{n-k}(z;q).
\end{equation} \end{prop}
\begin{proof}
From \eqref{A_n(z;q)}, \eqref{Def EX} and Definition \ref{q-Bernoulli} we get
\begin{equation*}\begin{split}
\sum_{n=0}^{\infty}\widetilde{B}_n(z;q) \frac{w^n}{[n]_q!} &= \dfrac{w\,exp_q(zw)\, exp_q(-w/2)}{exp_q(w/2)-exp_q(-w/2)} \\
    &=\sum_{n=0}^{\infty}\widetilde{A}_n(z;q) \frac{w^{n}}{[n]_q!}\,\sum_{n=0}^{\infty} \frac{q^{\frac{n(n-1)}{4}}}{[n]_q!} (\frac{-w}{2})^n  \\
    &= \sum_{n=0}^{\infty} \frac{w^{n}}{[n]_q!}\,
 \sum_{k=0}^{n}\qbinom{n}{k} (-\frac{1}{2})^k  q^{\frac{k(k-1)}{4}}\, \widetilde{A}_{n-k}(z;q).
\end{split}\end{equation*}
Comparing the coefficient of $\frac{w^n}{[n]_q!}$ we obtain the required result. \end{proof}
\begin{thm}\label{inverse of A_n}
Let $z\in \mathbb{C}$. Then, the polynomials $\widetilde{A}_n(z;q)$ can be represented in terms of the $q$-Bernoulli polynomials $\widetilde{B}_{n}(z;q)$ as \begin{equation}\label{Thm. An}
\widetilde{A}_{n}(z;q)= [n]_q!\sum_{j=0}^{n-1} (-\frac{1}{2})^{j+1}\,  \frac{\tilde{a}_j}{[n-j-1]_q!} \widetilde{B}_{n-j-1}(z;q),\end{equation}
where \begin{equation}\label{cons.}
\tilde{a}_j=  \sum_{k=0}^{j} (-1)^k \,\sum_{s_1+s_2+\ldots+s_k=n\atop  s_i>0\,(i=1,\ldots,k) } \dfrac{ q^{ \sum_{i=0}^k s_i(s_i+1)/4 }}{[s_1+1]_q! [s_2+1]_q!\ldots [s_{k}+1]_q!}.\end{equation}
\end{thm}
\begin{proof}
We can write the generating function of the $q$-polynomials $\widetilde{A}_n(z;q)$ as
\begin{equation*}\dfrac{w\,exp_q(zw)}{exp_q(w/2)-exp_q(-w/2)}= \dfrac{1}{exp_q(-w/2)}\,\Big[ \dfrac{w\, exp_q(zw)\, exp_q(-w/2)}{exp_q(w/2)-exp_q(-w/2)}\Big].\end{equation*}
Putting   $\displaystyle  a_n(k)= \sum_{s_1+s_2+\ldots+s_k=n\atop  s_i>0\,(i=1,\ldots,k) } \dfrac{ q^{ \sum_{i=0}^k s_i(s_i+1)/4 }}{[s_1+1]_q! [s_2+1]_q!\ldots [s_{k}+1]_q!}$, and then using Lemma \ref{inverse of exp.} we obtain
\begin{eqnarray*}
&&\sum_{n=0}^{\infty} \widetilde{A}_n(z;q)\frac{w^n}{[n]_q!}\\ &=& w\,\Big(\sum_{n=0}^{\infty} (\frac{-1}{2})^{n+1} [n]_q!\, \frac{w^n}{[n]_q!} \sum_{k=0}^{n}(-1)^ka_{n}(k)\Big)\,\Big( \sum_{n=0}^{\infty} \widetilde{B}_{n}(z;q) \frac{w^n}{[n]_q!}\Big) \\ &=&
w\,\sum_{n=0}^{\infty}\frac{w^n}{[n]_q!}\, \sum_{j=0}^{n}\qbinom{n}{j}[j]_q! (\frac{-1}{2})^{j+1} \sum_{k=0}^{j} a_j(k) (-1)^k   \widetilde{B}_{n-j}(z;q) \\ &=&
\sum_{n=0}^{\infty}\frac{w^{n+1}[n+1]_q}{[n+1]_q!}\, \sum_{j=0}^{n}\qbinom{n}{j}[j]_q! (\frac{-1}{2})^{j+1} \sum_{k=0}^{j} a_j(k) (-1)^k   \widetilde{B}_{n-j}(z;q).
\end{eqnarray*}
This implies
\begin{equation*} \widetilde{A}_{n+1}(z;q) = [n+1]_q\sum_{j=0}^{n}\qbinom{n}{j}[j]_q! (\frac{-1}{2})^{j+1} \sum_{k=0}^{j} a_j(k) (-1)^k   \widetilde{B}_{n-j}(z;q),\end{equation*} and then we obtain the required result.
\end{proof}
\begin{cor}\label{Cor.2}
For $n\in\mathbb{N}_0$ and $z\in \mathbb{C}$, the power series of the polynomial $\widetilde{A}_{n}(z;q)$ takes the form
\begin{equation*}
\widetilde{A}_{n}(z;q)= \sum_{m=0}^{n-1} \tilde{c}_m(n) \, \frac{z^m}{[m]_q!},
\end{equation*} where
\be \tilde{c}_m(n)= [n]_q!\, (\frac{-1}{2})^{n+1}\sum_{r=n-1}^{m}  q^{\frac{m(m-1)}{4}}\frac{(-2)^r\, \tilde{a}_r}{[r-m]_q!}\, \widetilde{\beta}_{r-m}(q).\ee
\end{cor}
\begin{proof}
From Theorem \ref{inverse of A_n} and Proposition \ref{prop:1-1}, we get
\begin{equation*}
\begin{split}
\widetilde{A}_{n}(z;q) &= [n]_q!\, (-\frac{1}{2})^{n}\sum_{r=0}^{n-1} \,  \frac{(-2)^r\,\tilde{a}_r}{[r]_q!} \widetilde{B}_{r}(z;q) \\
    &=[n]_q!\, (-\frac{1}{2})^{n}\sum_{r=0}^{n-1} \,  \frac{(-2)^r\,\tilde{a}_r}{[r]_q!} \sum_{m=0}^{r}\qbinom{r}{m} q^{\frac{m(m-1)}{4}}\, \widetilde{\beta}_{r-m}(q) z^m \\
    &= [n]_q!\, (-\frac{1}{2})^{n}\sum_{m=0}^{n-1} \Big(\sum_{r=n-1}^{m}  q^{\frac{m(m-1)}{4}}\frac{(-2)^r\, \tilde{a}_r}{[r-m]_q!}\, \widetilde{\beta}_{r-m}(q)\Big)\, \frac{z^m}{[m]_q!}.
\end{split}
\end{equation*}
\end{proof}
\begin{cor}\label{Sinh_q}
Let $w$ be a complex number such that $|w|< S_1$. Then
\begin{equation*}\frac{1}{Sinh_q(w)}= \sum_{n=0}^{\infty} d_n\, (2w)^{n},\end{equation*}
where $d_0=1$, $\displaystyle d_n= [n+1]_q \sum_{j=0}^{n} (-\frac{1}{2})^{j+1}\,  \frac{\tilde{a}_j}{[n-j]_q!} \widetilde{\beta}_{n-j}(q)$ and $\tilde{a}_j$ the constants defined in \eqref{cons.}.
\end{cor}
\begin{proof}
The proof follows immediately from \eqref{A_n(z;q)}, \eqref{Thm. An} and replacing $w$ by $2w$.
\end{proof}
\begin{rem}
According to the definition of $q$-Bernoulli numbers \eqref{qBernoulli-numbers}, we have $\widetilde{\beta}_{n}(q)= \widetilde{A}_{n}(-\frac{1}{2};q)$. That is, for $n\in \mathbb{N}_0$ we have
\begin{equation*}
\widetilde{\beta}_{n+1}(q)= [n+1]_q!\sum_{j=0}^{n} (-\frac{1}{2})^{j+1}\,  \frac{\tilde{a}_j}{[n-j]_q!} \widetilde{B}_{n-j}(-\frac{1}{2};q),\end{equation*}
where $\tilde{a}_j$ is the constants which defined in \eqref{cons.}.
\end{rem}
\section{{\bf A $q$-Euler polynomials generated by the third Jackson $q$-Bessel function }} \label{q-Euler}
\begin{defn} A $q$-Euler polynomials $\widetilde{E}_n(z;q)$ are defined by  the generating function
\be\label{Defn.Euler}
 \dfrac{2\,exp_q(zw)\, exp_q(\frac{-w}{2})}{exp_q(\frac{w}{2})+exp_q(\frac{-w}{2})}=\sum_{n=0}^{\infty}\widetilde{E}_n(z;q)
\frac{w^n}{[n]_q!},\ee
 and the $q$-Euler numbers $\widetilde{e}_n(q)$ are defined in terms of  generating function
 \be\label{qEuler-numbers}
\dfrac{2}{exp_q(w)+exp_q(-w)} =\sum_{n=0}^{\infty}\widetilde{e}_n(q)\frac{w^n}{[n]_q!}.\ee \end{defn}
\noindent Clearly, $\widetilde{e}_{2n+1}(q)=0$ for all $n\in\mathbb{N}_0$.
Consequently,
\be\label{Cosine}
\dfrac{1}{C_q (z)}=\sum_{n=0}^{\infty}(-1)^n
\dfrac{\widetilde{e}_{2n}(q)}{[2n]_q!}\, z^{2n}; \quad |z|<C_1,\ee
where $C_1$ is the first positive zeros of $C_q (z)$.

\noindent We use the notation $\widetilde{E}_n$ to denotes the first Euler number, i.e.,
\begin{equation*} \widetilde{E}_n := \widetilde{E}_n(0;q), \quad (n\in \mathbb{N}_0).\end{equation*}
\begin{prop}
The $q$-Euler  polynomials $\widetilde{E}_n(z;q)$  are given by
\[\widetilde{E}_0(z;q)=1,\]
and for $n\in\mathbb{N}$
\[\widetilde{E}_n(z;q)=\sum_{k=0}^{n}\qbinom{n}{k} q^{\frac{k(k-1)}{4}}
\widetilde{E}_{n-k}\, z^k.\]
\end{prop}
\begin{proof}
The proof is similar to  the proof of  Proposition~\ref{prop:1-1} and is omitted.
\end{proof}
\begin{prop}
For $n\in\mathbb{N}_0$, we have
\be\label{E(even)} \widetilde{E}_{n}(\frac{1}{2};q)=(\frac{1}{2})^{n} \sum_{n=0}^{n}\,\qbinom{n}{k}\, (-1)^k\,q^{\frac{k(k-1)}{4}}
(q^{\frac{1-k}{2}};q)_k\, \widetilde{e}_{n-k}(q).\ee
\end{prop}
\begin{proof}
Since
\begin{equation*}
exp_q(\frac{w}{2})\,exp_q(\frac{-w}{2})= \sum_{n=0}^{\infty} (\frac{-1}{2})^n q^{\frac{n(n-1)}{4}}\, (q^{\frac{1-n}{2}};q)_n \frac{w^n}{[n]_q!},
\end{equation*}
then, by using \eqref{Defn.Euler} and \eqref{qEuler-numbers} we get
\begin{equation*}   \sum_{k=0}^{\infty}  \widetilde{E}_{n}(\frac{1}{2};q)\, \frac{w^n}{[n]_q!}= \sum_{k=0}^{\infty} \frac{w^n}{[n]_q!} \sum_{k=0}^{n}(\frac{1}{2})^{n} \,(-1)^k\qbinom{n}{k}\,q^{\frac{k(k-1)}{4}}
(q^{\frac{1-k}{2}};q)_k\, \widetilde{e}_{n-k}(q), \end{equation*}
which implies the result.
 \end{proof}
Note that if $z=\frac{1}{2}$, then the left hand side of \eqref{Defn.Euler} is an even function. Hence,
\be\widetilde{E}_{2n+1}(\frac{1}{2};q)=0.\ee
\begin{prop}
For $n\in\mathbb{N}_0$, we have $\displaystyle \widetilde{E}_{2n}=\delta_{n,0}$\,,
where $\delta_{n,0}$ is the Kronecker's delta.
\end{prop}
\begin{proof}
Observe that
\be
\dfrac{2\,exp_q(\frac{-w}{2})}{exp_q(\frac{w}{2})+exp_q(\frac{-w}{2})}-1=
\dfrac{exp_q(\frac{-w}{2})-exp_q(\frac{w}{2})}{exp_q(\frac{w}{2})+exp_q(\frac{-w}{2})}.
\ee
So, we obtain
\be\label{Prop 3.1}
\sum_{n=0}^{\infty}\frac{\widetilde{E}_n}{[n]_q!} w^n
= 1+ \dfrac{exp_q(\frac{-w}{2})-exp_q(\frac{w}{2})}{exp_q(\frac{w}{2})+exp_q(\frac{-w}{2})}.\ee
 The right hand side of \eqref{Prop 3.1} is an odd function, therefore the even  powers of $w$ on the left hand side of this equation vanish.
Hence $\widetilde{E}_0=1$ and $\widetilde{E}_{2n}=0$ for every $n\in\mathbb{N}$.
\end{proof}
The following results can be proved by the same way of Proposition \ref{Ber.q and 1/q}, Theorem \ref{Eq.q-D(TH)} and Theorem \ref{sum:B}.
\begin{prop} For $n\in\mathbb{N}$ and $z\in \mathbb{C}$, we have
\begin{enumerate}
  \item $\displaystyle \widetilde{E}_n(z;q)= q^{\frac{n(n-1)}{2}}\widetilde{E}_n(z;1/q)$;
  \item $\displaystyle \frac{\delta_q \widetilde{E}_n(z;q)}{\delta_q z}=[n]_q\, \widetilde{E}_{n-1}(z;q)$.
\end{enumerate}
\end{prop}
\begin{thm}\label{Euler-Relation}
For $z\in\mathbb{C}$, we have the identities
 \begin{equation*}\begin{split}
\frac{q^{\frac{n(n-1)}{4}}}{[n]_q!}\,(\frac{-1}{2})^{n} (2q^{\frac{1-n}{2}}\, z;q)_n &= \sum_{k=0}^{[\frac{n}{2}]}(\frac{1}{2})^{2k} \frac{q^{\frac{k(2k-1)}{2}}}{[2k]_q!}\, \frac{\widetilde{E}_{n-2k}(z;q)}{[n-2k]_q!};\\
\frac{q^{\frac{n(n-1)}{4}}}{[n]_q!}\,(\frac{-1}{2})^{n} &= \sum_{k=0}^{[\frac{n}{2}]}(\frac{1}{2})^{2k} \frac{q^{\frac{k(2k-1)}{2}}}{[2k]_q!}\, \frac{\widetilde{E}_{n-2k}}{[n-2k]_q!}.\end{split}\end{equation*}
\end{thm}
As a consequence of Theorem \ref{Euler-Relation}, we get
 \begin{equation*}\begin{gathered} \widetilde{E}_{0}=1, \quad \widetilde{E}_{1}= -\frac{1}{2}, \quad \widetilde{E}_{2}=0, \quad
  \widetilde{E}_{3}= \dfrac{(q-1)q^{\frac{3}{2}}+(1-q^3)\,q^{\frac{1}{2}}}{8(1-q)},\\
 \widetilde{E}_{5}=\dfrac{(q^2-1)q^5+(1-q^2)[5]_q\,q^3+ (q-1)q^{\frac{1}{2}}[4]_q[5]_q([3]_q q^{\frac{1}{2}}-{\frac{3}{2}})}{32(1-q^2)}.
 \end{gathered}\end{equation*}
\begin{prop}
 We have the identity
\begin{equation}\label{tansh}Tanh_q(\frac{w}{2})=\sum_{n=0}^{\infty} \widetilde{E}_{2n+1} \frac{w^{2n+1}}{[2n+1]_q!}, \quad |\frac{w}{2}|< S_1.
\end{equation}\end{prop}
\begin{proof}
The proof is similar to  the proof of  Proposition~\ref{coth_q} and is omitted.
\end{proof}
\vskip5mm

Recall that the $q$-tangent and $q$-secant numbers defined by  the series expansions of $Tan_q z$ and $Sec_q z$  by
\be\label{tan:exp}\begin{split}Tan_qz &=\sum_{n=0}^{\infty}T_{2n+1}(q)\frac{z^{2n+1}}{[2n+1]_q!},\\ Sec_qz &=\frac{1}{C_q u}=\sum_{n=0}^{\infty}S_{2n}(q)\frac{z^{2n}}{[2n]_q!},\end{split}\ee
(for more details see~\cite{Foata,Huber}).

\noindent  Consider $S_q(z)$ and $C_q(z)$ which defined in \eqref{S&C}. Then, from \eqref{Cosine} and \eqref{tansh} we get
 \[T_{2n+1}(q)=(-1)^{n}\widetilde{E}_{2n+1} 2^{2n+1}, \quad S_{2n}(q)=(-1)^n \widetilde{e}_{2n}(q).\]
\begin{thm}
 For $n\in\mathbb{N}_0$
 \be\label{Convol.}\sum_{k=0}^{n}(-1)^k\, 2^{2k} \frac{\beta_{2k}(q)}{[2k]_q!}\frac{T_{2n-2k+1}(q)}{[2n-2k+1]_q!}=\delta_{0,n}\,,\ee
 where $\delta_{n,0}$ is the Kronecker's delta.
\end{thm}
\begin{proof}
 From Proposition \ref{coth_q}, we have
 \be\label{zcot:exp} z\, Cot_q(z)=\sum_{n=0}^{\infty}(-1)^n\, 2^{2n}\beta_{2n}(q)\frac{z^{2n}}{[2n]_q!}.\ee Observe that
  $z\, Tan_q (z)\, Cot_q(z)= z$. So, by using \eqref{tan:exp} and \eqref{zcot:exp} we obtain
  \[z=\sum_{n=0}^{\infty} (-1)^n\, 2^{2n}\beta_{2n}(q)\frac{z^{2n}}{[2n]_q!}\sum_{n=0}^{\infty} T_{2n+1}(q)\,\frac{z^{2n+1}}{[2n+1]_q!}.\]
Therefore,
\[\sum_{n=0}^{\infty} z^{2n}\sum_{k=0}^{n}(-1)^k 2^{2k} \frac{\beta_{2k}(q)}{[2k]_q!}\frac{T_{2n-2k+1}(q)}{[2n-2k+1]_q!}=1.\]
Comparing the coefficient of $z^{2n}$, we obtain the desired result.
\end{proof}
\begin{cor}
Let $n\in\mathbb{N}_0$. Then, the $q$-tangent numbers $T_{2n+1}(q)$ are positive numbers.
\end{cor}
\begin{proof}
From Equation \eqref{Convol.}, we get
 \[\frac{T_{2n+1}}{[2n+1]_q!}=\sum_{k=1}^{n}(-1)^{k-1} (2)^{2k} \frac{\beta_{2k}(q)}{[2k]_q!}\frac{T_{2n-2k+1}(q)}{[2n-2k+1]_q!}.\]
Since $(-1)^{k-1}\beta_{2k}>0$ for $k\in\mathbb{N}$ and  $T_1(q)=1>0$, then we can prove the result by induction on $n$ for all $n\in\mathbb{N}_0$.
\end{proof}
\vskip5mm
\noindent We define a sequence of polynomials $\widetilde{M}_n(z;q)$ by the generating function
\be\label{M_n(z;q)} \dfrac{exp_q(zw)}{exp_q(w/2)+exp_q(-w/2)}=\sum_{n=0}^{\infty}\widetilde{M}_n(z;q) \frac{w^n}{[n]_q!}.\ee
Similarly to Proposition \ref{Prop:B&A}, Theorem \ref{inverse of A_n} and  Corollary \ref{Sinh_q}, we have the following results.
\begin{prop}\label{prop:E}
For $n\in\mathbb{N}$, the $q$-Euler polynomials $\widetilde{E}_n(z;q)$ can be represented in terms of $\widetilde{M}_{n}(z;q)$ as
\begin{equation}\label{prop:1-2}
\widetilde{E}_n(z;q)= \sum_{k=0}^{n}\qbinom{n}{k} \, (-1)^k\, (\frac{q^{k/4}}{2})^{k-1}\, \widetilde{M}_{n-k}(z;q).
\end{equation} \end{prop}
\begin{thm}\label{Lem:2}
 For $n\in\mathbb{N}$ and $z\in\mathbb{C}$,  $\widetilde{M}_n(z;q)$ can be represented in terms of the $q$-Euler polynomials $\widetilde{E}_{n}(z;q)$ as
\begin{equation}
\widetilde{M}_{n}(z;q)= \frac{[n]_q!}{2}\sum_{j=0}^{n} (-\frac{1}{2})^{j+1}\,  \frac{\tilde{a}_j}{[n-j]_q!} \widetilde{E}_{n-j}(z;q),\end{equation}
where\begin{equation*}
\tilde{a}_j=  \sum_{k=0}^{j} (-1)^k \,\sum_{s_1+s_2+\ldots+s_k=n\atop  s_i>0\,(i=1,\ldots,k) } \dfrac{ q^{ \sum_{i=0}^k s_i(s_i+1)/4 }}{[s_1+1]_q! [s_2+1]_q!\ldots [s_{k}+1]_q!}.\end{equation*}
\end{thm}
\begin{cor}
For $n\in\mathbb{N}_0$ and $z\in\mathbb{C}$, the power series of the polynomial $\widetilde{M}_{n}(z;q)$ takes the form
\begin{equation*}
\widetilde{M}_{n}(z;q)= \sum_{m=0}^{n} c_m(n) \, \frac{z^m}{[m]_q!},
\end{equation*} where
\begin{equation*} c_m(n)= \frac{[n]_q!}{2}\, (\frac{-1}{2})^{n+1}\sum_{r=n}^{m}  q^{\frac{m(m-1)}{4}}\frac{(-2)^r\, \tilde{a}_r}{[r-m]_q!}\, \widetilde{E}_{r-m}.\end{equation*}
\end{cor}
\begin{prop}
For $z\in\mathbb{C}$, we have
\begin{equation*}\frac{1}{Cosh_q(\frac{w}{2})}= \sum_{n=0}^{\infty} \widetilde{d}_n\, w^{n}, \quad  |\frac{w}{2}|< C_1,\end{equation*}
where  $\displaystyle \widetilde{d}_n= \sum_{j=0}^{n} (-\frac{1}{2})^{j+1}\,  \frac{\tilde{a}_j}{[n-j]_q!} \widetilde{E}_{n-j}$ and $\tilde{a}_j$ the constants defined in \eqref{cons.}.
\end{prop}
\section{{\bf A $q$-Lidstone series involving $q$-Bernoulli polynomials}} \label{q-Lidstone-S}

Our aim of this section is to prove that an entire function $f$ may be expanded  in terms of $q$-Lidstone polynomials, where the coefficients of these  polynomials are the even powers of the $q$-derivative $\frac{\delta_q f(z)}{\delta_q z}$ at $0$ and $1$.
\vskip3mm

We begin by recalling some definitions and results from \cite{Ramis} which will be used in the proof of the main result.
\begin{defn}\label{def:Ramiss} Let $k$ be a non zero real number, and let $p$ be a real number with $|p|>1$.  An entire function $f$ has a $p$-exponential growth of order $k$ and a finite type, if there exist real numbers $K>0$ and $\alpha$, such that
\[|f(z)|< K p^{\frac{k}{2}\left(\dfrac{log |z|}{log \, p}\right)^2} |z|^{\alpha},\]
or equivalently,
\[|f(z)|\leq K e^{\frac{k}{2\log p}(\log |z|)^2+\alpha \log |z|}.\]
\end{defn}
\begin{defn}
Let $k$ be a non zero real number and let $p$ be a real number, with
 $|p|>1$. A formal power series expansion  $ \hat{f}:=\sum_{=0}^{\infty} a_n z^n$ is $p$-Gevery of order $\frac{1}{k}$ (or of level $k$), if there exists real numbers $C$, $A>0$ such that
\[|a_n|<Cp^{\frac{n(n+1)}{2k} }\, A^n.\]
\end{defn}
\begin{prop}\label{prop:Ramiss}
Let $k$ be a non zero real number and  $p$ be a real number, with
$p>1$. The following statements are equivalent.
\begin{itemize}
\item[i]  The series $\hat{f}:=\sum_{n=0}^{\infty}a_n x^n$ is $p$-Gevery of
order $-k$;
\item[ii] The series $\hat{f}$ is the power series expansion at the origin of
an entire function $f$ having a $p$-exponential growth of order $k$ and
a finite type $\alpha$, where
\[|a_n|<Ke^{-\frac{(n-\alpha)^2}{2k} },\quad K>0.\]
\end{itemize}
\end{prop}
\begin{rem}
The series $\sum_{n=0}^\infty\, \frac{q^{\frac{n(n-1)}{4}}}{[n]_q!}\, z^n$ which defines the function $exp_q(z)$ is $q^{-1}$- Gevery of order $-2$. Consequently, $exp_q(z)$ has $q^{-1}$- exponential growth of order $2$.
\end{rem}

\vskip5mm
\begin{prop}\label{prop:An}  Let $z $ and $w$ be complex numbers
such that $|w|<S_1$. Then
\be\label{Eq:Ln}
Sinh_q(wz)\, Csch_q(w)
=\sum_{n=0}^{\infty} \frac{2^{2n+1}}{[2n+1]_q!} \widetilde{A}_{2n+1}(z/2;q)\, w^{2n},\ee
where $\widetilde{A}_{n}(z;q)$ are the $q$-polynomials defined in \eqref{A_n(z;q)}.
\end{prop}
\begin{proof}First, note that the function $g_q(z,w):= Sinh_q(wz)\, Csch_q (w)$ is holomorphic for $|w|<S_1$. By using \eqref{sh+cosh},  we can write
\[ g_q(z,w):= \dfrac{exp_q(zw)-exp_q(-zw)}{exp_q (w)-exp_q (-w)}.\]
 Then, by using \eqref{A_n(z;q)} we get
\begin{eqnarray*}
g_q(z,w)&:=&\dfrac{exp_q(zw)-exp_q(-zw)}{exp_q(w)-exp_q(-w)}\\
&=& \frac{1}{2w} \dfrac{2w\, exp_q(zw)}{exp_q(w)-exp_q(-w)} - \frac{1}{2w}\dfrac{2w\, exp_q(-zw)}{exp_q(-w)-exp_q(w)}\\
&=& \frac{1}{2w} \sum_{n=0}^{\infty}  \widetilde{A}_{n}(z/2;q)\frac{(2w)^n}{[n]_q!}- \frac{1}{2w}\sum_{n=0}^{\infty}  \widetilde{A}_{n}(z/2;q)\frac{(-2w)^n}{[n]_q!}\\
&=&\sum_{n=0}^{\infty} \frac{2^{2n+1}}{[2n+1]_q!} \widetilde{A}_{2n+1}(z/2;q)\, w^{2n}.
\end{eqnarray*}
\end{proof}
Henceforth, we will consider the notation \be\label{A_n2}\widetilde{A}_n(z)=\frac{2^{2n+1}}{[2n+1]_q!} \widetilde{A}_{2n+1}(z/2;q).\ee
So, the previous result can be restated in the following form:
\be\label{Eq:An2}
\dfrac{exp_q(zw)-exp_q(-zw)}{exp_q(w)-exp_q(-w)}=\sum_{n=0}^{\infty} \widetilde{A}_{n}(z) w^{2n},\ee
\begin{cor}\label{cor:A}
For $n\in\mathbb{N}$, the $q$-polynomials  $\widetilde{A}_n(z)$ satisfy the\
 $q$-difference equation
\[   \frac{\delta^{2}_q \,\widetilde{A}_n(z) }{\delta_q z^2}  =\widetilde{A}_{n-1}(z),\]
with the boundary conditions $\widetilde{A}_n(0)=\widetilde{A}_n(1)=0$, and $\widetilde{A}_0(z)=z$.
\end{cor}
\begin{proof}
By using \eqref{delta E} we obtain
\begin{eqnarray*} \frac{\delta^{2}_q \,g(z,w)}{\delta_q z^2} &=& \sum_{n=0}^{\infty} \frac{\delta^{2}_q \,\widetilde{A}_n(z) }{\delta_q z^2}\, w^{2n}\\ &=& w^2 \dfrac{exp_q(zw)-exp_q(-zw)}{exp_q(w)-exp_q(-w)} \\ &=& \sum_{n} \widetilde{A}_n(z)\, w^{2n+2}.\end{eqnarray*}
Therefore, $  \frac{\delta^{2}_q \,\widetilde{A}_n(z) }{\delta_q z^2}  =\widetilde{A}_{n-1}(z) \quad (n\in\mathbb{N})$. Furthermore,
$$ \widetilde{A}_{0}(z)=\lim_{w\rightarrow 0} \dfrac{exp_q(zw)-exp_q(-zw)}{exp_q(w)-exp_q(-w)} =z.$$
Substitute with $z=0$ and $z=1$ in Equation \eqref{Eq:An2}, we obtain
 $$\widetilde{A}_n(0)=\widetilde{A}_n(1)=0 \, \mbox{ for all } \, n\in\mathbb{N}.$$
\end{proof}
\begin{prop}\label{prop:2} Let $z $ and $w$ be complex numbers such that $|w|<S_1$. Then
\be
\dfrac{exp_q (zw)exp_q(-w)-exp_q(-zw)exp_q(w)}{exp_q(w)-exp_q(-w)}
=\sum_{n=0}^{\infty} \widetilde{B}_n(z)\,w^{2n},\ee
where
\[\widetilde{B}_n(z)=\frac{2^{2n+1}}{[2n+1]_q!}\widetilde{B}_{2n+1}(z/2;q).\]
\end{prop}
\begin{proof}If  $z $ and $w$ are complex numbers such that
$|w|<S_1$, then
\begin{eqnarray*}
&&\dfrac{exp_q(zw)exp_q(-w)-exp_q(-zw)exp_q(w)}{exp_q(w)-exp_q(-w)}\\
&=&\frac{1}{2w}\left[\dfrac{2w\, exp_q(zw)exp_q(-w)}{ exp_q(w)-exp_q(-w)} \right]-
\frac{1}{2w}\left[\dfrac{2w\, exp_q(-zw)exp_q(w)}{ exp_q(w)-exp_q(-w)}\right]\\
&=&\frac{1}{2w}\sum_{n=0}^{\infty}\dfrac{(2w)^{n}-(-2w)^n}{[n]!}
\widetilde{B}_{n}(z/2;q)\\
&=&\sum_{n=0}^{\infty}\frac{w^{2n}}{[2n+1]!}2^{2n+1}\widetilde{B}_{2n+1}(z/2;q).
\end{eqnarray*}
\end{proof}
As in Corollary~\ref{cor:A}, one can verify that $\widetilde{B}_0(z)=z-1$ and
for $n\in\mathbb{N}$, the $q$-polynomials  $\widetilde{B}_n(z)$ satisfy the\
 $q$-difference equation
\[   \frac{\delta^{2}_q \,\widetilde{B}_n(z) }{\delta_q z^2}  =\widetilde{B}_{n-1}(z),\]
with the boundary conditions $\widetilde{B}_n(0)=\widetilde{B}_n(1)=0$.
\vskip6mm
Now, observe that
\[\begin{gathered}exp_q(zw) =\dfrac{exp_q(zw)exp_q(-w)-exp_q(-zw)exp_q(w)}{exp_q(-w)-exp_q(w)}\\
+\,exp_q(w)\,\dfrac{exp_q(zw)-exp_q(-zw)}{exp_q(w)-exp_q(-w)}.\end{gathered}\]
So, from Proposition \ref{prop:An} and Proposition \ref{prop:2}  we get immediately the following result.
\begin{prop}\label{Prop:3} If $z$ and $w$ are complex numbers such that  $|w|<S_1$, then
\be
 exp_q(zw)=exp_q(w)\, \sum_{n=0}^{\infty} \widetilde{A}_n(z)w^{2n} -\sum_{n=0}^{\infty}\widetilde{B}_n(z)w^{2n}.
\ee
\end{prop}

\vskip6mm
In the following, we assume that $\Psi$ is a comparison function, i.e. $\Psi(t)=\sum_{n=0}^{\infty}\Psi_n t^n$ such that
$\Psi_n>0$ and $\Big(\Psi_{n+1}/\Psi_n\Big) \downarrow 0$ (see \cite{Boas-Buck,Nachbin}). We denote by
$\mathcal{R}_{\Psi}$  the class of all entire functions $f$ such that, for some numbers $\tau$,
\begin{equation}\label{tau}|f(re^{i\theta})|\leq M \Psi(\tau \,r),\end{equation}
as $r\rightarrow \infty$. Here, the complex variable $z$ was written as $z = r e^{i\theta}$ to emphasize that the limit must hold in all directions $\theta$. The infimum of numbers $\tau$ for which \eqref{tau} holds is the $\Psi$-type of the function $f$. This type can be computed by applying Nachbin's  theorem \cite{Nachbin} which states that a function $f(z)=\sum_{n=0}^{\infty}
f_n z^n$ is of $\Psi$-type $\tau  $ if and  only if
$$\tau= \limsup_{n\rightarrow \infty}  \Big|\frac{f_n}{\Psi_n}\Big|^{\frac{1}{n}}.$$

\noindent In~\cite{Boas-Buck}, the authors applied Nachbin's theorem for the generalized Borel transform \begin{equation*}
F(w)=\sum_{n=0}^{\infty} \dfrac{f_n}{\Psi_n w^{n+1}},\end{equation*}
and they  proved the following result.
\begin{thm}\label{thm-Kernel expansion}
Let $f(z)$ belong to the class
$\mathcal{R}_{\Psi}$, and let $D(f)$ be the closed set consists of the
union  of the  set of all singular points  of $F$ and the set of all points
exterior to  the domain of $F$. Then
\[f(z)=\frac{1}{2\pi i} \int_{\Gamma} \Psi(zw) F(w)\,dw\]
where $\Gamma$ encloses $D(f)$.
\end{thm}

According to the above arguments  and results we will prove the main theorem.

\begin{thm}\label{Thm:qLidstone expansion 2}
Let $S_1$ be the smallest positive zero of $S_q(z)$.  Assume that one of the following conditions hold:
\begin{enumerate}
\item[(i)] The function   $f(z)$ is an entire function of
$q^{-1}$-exponential growth of order $2$  and a finite type $\alpha$, where
\begin{equation}\label{Co. alpha}
\alpha < 2\,\Big(\frac{1}{4}- \frac{\log S_1}{\log q}\Big).
\end{equation}
\noindent \item[(ii)]  The function $f(z)$ is an entire function of $q^{-1}$-
 exponential  growth  of order less than $2$.
\end{enumerate}
Then  $f(z)$ has a convergent $q$-Lidstone representation
\[f(z)=\sum_{n=0}^{\infty}\left[\widetilde{A}_n(z)\,\frac{\delta^{2n}_q \, f(1)}{\delta_q z^{2n}}-\widetilde{B}_n(z)\,
\frac{\delta^{2n}_q \, f(0)}{\delta_q z^{2n}}\right],\]
where $\widetilde{A}_n(z)$ is the polynomial of degree $2n+1$ defined in
\eqref{A_n2} and
\[\widetilde{B}_n(z):=\frac{2^{2n+1}}{[2n+1]_q!}\widetilde{B}_{2n+1}(z/2;q).\]
\end{thm}
\begin{proof}
We apply Theorem~\ref{thm-Kernel expansion} when $\Psi(z)$ chosen as $exp_q(z)$ and
 \[ \Psi_n=\frac{q^{\frac{n(n-1)}{4}}}{[n]_q!}.\]  Notice, the sequence  \[\dfrac{\Psi_{n+1}}{\Psi_n}=\frac{q^{n/2}(1-q)}{1-q^{n+1}}= \frac{q^{n/2}}{[n+1]_q}\]
is decreasing and vanishes at $\infty$. By using Proposition \ref{prop:Ramiss}, we have for any entire function  $f(z)=\sum_{=0}^{\infty} a_n z^n$ of $q^{-1}$- exponential growth of order $k$ and a finite type $\alpha$, there exists a real number $K>0$ such that
\[|a_n|\leq Kq^{\frac{(n-\alpha)^2}{2k}}.\]
According to the assumption, we have two cases:

\noindent Case 1. If $k=2$, then $|a_n|\leq Kq^{\frac{(n-\alpha)^2}{4}}$. This implies \eqref{tau} holds and  $f\in \mathcal{R}_{\Psi}$. Here, the $\Psi$-type of the function $f$ given by
\begin{eqnarray*}
\tau &:=&\limsup_{n\rightarrow \infty}  \Big|\frac{a_n}{\Psi_n}\Big|^{\frac{1}{n}} \\
&\leq& \frac{q^{\frac{1}{4}- \alpha/2}}{(1-q)} \limsup_{n\rightarrow \infty}  \Big( K\,(q;q)_n q^{\alpha^2/4}\Big)^{\frac{1}{n}}\\
 &\leq&   q^{\frac{1}{4}- \alpha/2}< S_1.\end{eqnarray*}

\noindent Case 2. If $k<2$, then $\tau=0$.

\noindent So, we can take  $D(f)$ lies in the closed disk $|w|\leq \tau \leq q^{\frac{1}{4}- \alpha/2} < S_1$ and take the curve $\Gamma$  as the circle $|w|=\tau+\epsilon <S_1$, $\epsilon>0$ which encloses $D(f)$. Note that the inequality $q^{\frac{1}{4}- \alpha/2} < S_1$ satisfies the condition \eqref{Co. alpha} on the type of $f(z)$. We obtain
\[f(z)=\frac{1}{2\pi i} \int_{\Gamma} exp_q(zw) F(w)\,dw.\]
Therefore,
\begin{eqnarray*}
    \frac{\delta^{2n}_q \, f(0)}{\delta_q z^{2n}}&=&\frac{1}{2\pi i}
\int_{\Gamma} \frac{\delta^{2n}_q}{\delta_q z^{2n}}\, exp_q (zw)|_{z=0}\, F(w)\,dw\\
&=&\frac{1}{2\pi i} \int_{\Gamma}  w^{2n}\, F(w)\,dw ,\\
\frac{\delta^{2n}_q \, f(1)}{\delta_q z^{2n}}&=&\frac{1}{2\pi i}\int_{\Gamma}  w^{2n}\,exp_q (w) \,F(w)\,dw
\end{eqnarray*}
Now, by using  Proposition \ref{Prop:3} we have
 \begin{eqnarray*}
&& f(z)= \frac{1}{2\pi\,i}\int_{\Gamma}exp_q(zw)\,F(w)\,dw \\
 &=&\frac{1}{2\pi\,i}\int_{\Gamma}\left\{ exp_q(w)\, \sum_{n=0}^{\infty} \widetilde{A}_n(z)w^{2n} -\sum_{n=0}^{\infty}\widetilde{B}_n(z)w^{2n}     \right\} F(w)\,dw\\
  &=& \sum_{n=0}^{\infty}\left[\widetilde{A}_n(z)\,\frac{\delta^{2n}_q \, f(1)}{\delta_q z^{2n}}-\widetilde{B}_n(z)\,
\frac{\delta^{2n}_q \, f(0)}{\delta_q z^{2n}} \right].
  \end{eqnarray*}
  \end{proof}
\begin{rem}\label{Rem.1}
In Theorem \ref{Thm:qLidstone expansion 2}, it is obvious if
$$ \frac{\delta^{2n}_q \, f(0)}{\delta_q z^{2n}}= \frac{\delta^{2n}_q \, f(1)}{\delta_q z^{2n}}=0, \,\, (n\in \mathbb{N})$$ then $f(z)$ is identically zero.
\end{rem}
The following example shows that the sign of equality can not be admitted in \eqref{Co. alpha}.
\begin{exa}
 Consider $f(z)= S_q(S_1z)$. Then $f$ is an entire function of
$q^{-1}$-exponential growth of order $2$  and a finite type $\alpha= \frac{1}{2}-2\, \frac{\log_q S_1}{\log_q q}$. By using \eqref{delta Sine and Cosine}, one can verify that
$$\frac{\delta^{2n}_q \, f(0)}{\delta_q z^{2n}}= \frac{\delta^{2n}_q \, f(1)}{\delta_q z^{2n}}=0\,\, (n\in \mathbb{N}).$$
This implies the $q$-Lidstone expansion of $f(z)$ vanishes identically but the function does not.
\end{exa}
\vskip 5mm
We end this section by given the $q$-Lidstone series of the functions $(z;q)_n$.
\begin{exa}
Consider the functions $g_n(z)=(z;q)_n$, $n\in \mathbb{N}$. Then, Condition (ii) of Theorem \ref{Thm:qLidstone expansion 2} is satisfied. So, these polynomials have a $q$-Lidstone representation
\[g_n(z)=\sum_{m=0}^{n}\left[\widetilde{A}_m(z)\,\frac{\delta^{2m}_q \,g_n(1)}{\delta_q z^{2m}}-\widetilde{B}_m(z)\,
\frac{\delta^{2m}_q \, g_n(0)}{\delta_q z^{2m}}\right].\]
One can verify that
\[\begin{gathered} \frac{\delta^{2m}_q \, g_n(z)}{\delta_q z^{n}}= q^{-m(m+\frac{1}{2})}\, [n]_q[n-1]_q...[n-2m+1]_q (q^mz;q)_{n-2m}.\end{gathered}\]
Therefore, $g_n(z)$  have the convergent $q$-Lidstone representation
\[\begin{gathered}
g_n(z)=\\ \sum_{m=0}^{n} q^{-m(m+\frac{1}{2})}\, [n]_q\,[n-1]_q\, ...\, [n-2m+1]_q  \left[(q^m;q)_{n-2m}\, \widetilde{A}_m(z) -\widetilde{B}_m(z)\right].\end{gathered}\]
\end{exa}

\section{{\bf A $q$-Lidstone series involving $q$-Euler polynomials}}
In this section, we introduce  another $q$-extension of Lidstone theorem. We expand the function in $q$-Lidstone polynomials which are $q$-Euler polynomials $\widetilde{E}_n(z;q)$  defined by the generating function \eqref{Defn.Euler}.

\noindent All the results can be studied in the same manner of the results of the previous section.
\begin{prop}\label{prop:Mn} If $z $ and $w$ are complex numbers
such that $|w|<C_1$, then
\be\label{Eq:Mn}
Cosh_q(wz)\, Sech_q(w)
=\sum_{n=0}^{\infty}\widetilde{M}_n(z) w^{2n},\ee
where \be\label{Def:Mn}\widetilde{M}_n(z):=\frac{2^{2n}}{[2n]_q!} \widetilde{M}_{2n}(z/2;q),\ee
and $\widetilde{M}_{n}(z;q)$ are the $q$-polynomials defined in \eqref{M_n(z;q)}.
\end{prop}
\begin{prop}\label{prop:2E} If  $z $ and $w$ are complex numbers such
that $|w|<C_1$, then
\be\begin{gathered}
\dfrac{exp_q(zw)exp_q(-w)-exp_q(-zw)exp_q(w)}{exp_q(w)+exp_q(-w)}\\
=\sum_{n=0}^{\infty}\frac{w^{2n+1}}{[2n+1]_q!}2^{2n+1}\widetilde{E}_{2n+1}(z/2;q).
\end{gathered}\ee
\end{prop}

\begin{prop}\label{Prop:3E} If $z$ and $w$ are complex numbers such
that  $|w|<C_1$,  then
\be
 exp_q(zw)=exp_q(w)\, \sum_{n=0}^{\infty} \widetilde{M}_n(z)w^{2n} -\sum_{n=0}^{\infty}\widetilde{N}_{n+1}(z)w^{2n+1},
\ee where
\[\widetilde{N}_{n+1}(z)=\frac{2^{2n+1}}{[2n+1]_q!}\widetilde{E}_{2n+1}(z/2;q).\]
\end{prop}
\begin{thm}
Assume that one of the following conditions hold:
\begin{enumerate}
\item[(i)] The function   $f(z)$ is an entire function of
$q^{-1}$-exponential growth of order $2$  and a finite type $\alpha$,
where \begin{equation}\label{Co. alpha C}\alpha< 2 \left(\frac{1}{4}-\frac{\log C_1}{\log q}\right);
 \end{equation}
\noindent \item[(ii)]  The function $f(z)$ is an entire function of $q^{-1}$-exponential  growth  of order less than $2$.
\end{enumerate}
Then  $f(z)$ has the convergent representation
\[f(z)=\sum_{n=0}^{\infty}\left[\widetilde{M}_n(z)\frac{\delta^{2n}_q \, f(1)}{\delta_q z^{2n}}-\widetilde{N}_{n+1}(z)
\frac{\delta^{2n+1}_q \, f(0)}{\delta_q z^{2n+1}}\right],\]
where $\widetilde{M}_n$ is the polynomial defined in \eqref{Def:Mn} and
\[\widetilde{N}_{n+1}(z):=\frac{2^{2n+1}}{[2n+1]_q!}\, \widetilde{E}_{2n+1}(z/2;q).\]
\end{thm}
As in Remark \ref{Rem.1}, the sign of equality can not be admitted in \eqref{Co. alpha C}. For example, the function $f(z)=
 \text{C}_q (C_1z)$ is a function of type $\left(\frac{1}{2}-2\frac{\log C_1}{\log q}\right)$ and one can verify that
 \[\frac{\delta^{2n}_q \, f(1)}{\delta_q z^{2n}} = 0 = \frac{\delta^{2n+1}_q \, f(0)}{\delta_q z^{2n+1}}.\]
 Hence, the $q$-Lidstone expansion of $f(z)$ vanishes while the function does not.
\section{\bf{ Concluding Remarks}}
The $q$-Lidstone's series approximates an entire function in a neighborhood of two points in terms of $q$-analog of Lidstone polynomials. In \cite{Ismail and Mansour}, the authors introduced these polynomials which were $q$-Bernoulli polynomials generated by second Jackson $q$-Bessel function.

In this paper, we presented $q$-Bernoulli and $q$-Euler polynomials generated by the third Jackson $q$-Bessel function to construct new types of $q$-Lidstone expansion theorem \cite{Ismail and Mansour}.

 This work provides the basis for several applications that we can search in the future. Firstly, we are interested in studying the generalization of $q$-Lidstone's series. The analogous problem for the classical case  was studied in \cite{Whittaker} by Whittaker. Secondly, we are interested in constructing the $q$-Fourier series for the $q$-Lidstone polynomials $\widetilde{A}_n(z)$ and  $\widetilde{B}_n(z)$, and applying such expansions to a solution of certain $q$-boundary value problems as in \cite{Mansour and AL-Towaileb} and \cite{Mansour and AL-Towaileb 2}.



\begin{thebibliography}{1}

 \bibitem{AMbook}
M. H. Annaby and Z. S. Mansour: \emph{$q$-Fractional Calculus and Equations}. \newblock{Lecture Notes in Mathematics  2056},
Springer-Verlag, Berlin (2012).

\bibitem{Boas-Buck}
R. P. Boas and R. C. Buck:\emph{ Polynomial expansions of analytic functions}.
\newblock Springer-Verlag, Berlin, second edition (1964).

\bibitem{BC}
J. ~Bustoz and J. L. Cardoso: Basic analog of {Fourier} series on a $q$-linear grid,
\emph{J. Approx. Theory}, {\bf 112}, 154-157 (2001).

\bibitem{Cardoso011}
J. L. Cardoso: Basic {Fourier} {series}: convergence on and outside the $q$-linear
  grid, \emph{ J. Fourier Anal. Appl.}, {\bf 17}(1), 96-114 (2011).

\bibitem{Foata} D. Foata and G. Han: The $q$-tangent and $q$-secant numbers via basic Eulerian
polynomials, \emph{Proc. Amer. Math. Soc.}, \textbf{138}, 385-393 (2010).

\bibitem{GR}
G.~Gasper and M.~Rahman: \emph{Basic Hypergeometric Series}. \newblock{Cambridge  university Press},
 second addition, Cambridge (2004).

\bibitem{Huber} T. Huber and A.J. Yee: Combinatorics of generalized $q$-Euler numbers, \emph{Journal of
Combinatorial Theory}, 361-388 (2010).

\bibitem{Ismail and Mansour} M. Ismail and Z. Mansour: $q$-analogs of Lidstone expansion theorem, two point Taylor expansion theorem, and Bernoulli polynomials, \emph{J. Analysis and applications}, Doi.org/10.1142/S0219530518500264 (2018). \vskip1mm

\bibitem{Jacson 1} F. Jackson: On $q$-functions and a certain difference operator, \emph{Trans. Roy. Soc. Edinburgh}, \textbf{46}, 64-72 (1908).

\bibitem{Koelink and Swarttouw}  H.T. Koelink and R.F. Swarttouw: On the zeros of the Hahn-Exton $q$-Bessel function and associated $q$-Lommel polynomials, \emph{J. Math. Anal. Appl.}, {\bf 186}:690--710, (1994).

 \bibitem{Lidstone} G. Lidstone: Notes on the extension of Aitken's theorem (for polynomial interpolation) to the Everett types, \emph{Proc. Edinb. Math. Soc.}  $\textbf{2}$, 16-19 (1929).\vskip1mm

\bibitem{Mansour and AL-Towaileb} Z. Mansour and M. AL-Towailb: $q$-Lidstone polynomials and existence results for $q$-boundary value problems, \emph{Boundary Value Problems}  2017:178, doi: 10.1186/s13661-017-0908-4 (2017).

\bibitem{Mansour and AL-Towaileb 2} Z. Mansour and M. AL-Towailb: The Complementary $q$-Lidstone Interpolating Polynomials and Applications,
\emph{Math. Comput. Appl.}, {\bf 25(2)}, 34; doi: 10.3390/mca25020034 (2020).

\bibitem{Nachbin}
L.~Nachbin: An extension of the notion of integral  functions of finite exponential  type, \emph{Anais Acad. Brasil. Ciencias}
{\bf 16}: 143-147 (1944).

\bibitem{Nalci-Pashaev}
S. Nalci  and  O. Pashaev: \newblock{$q$-Bernoulli numbers and zeros of $q$-sine function}. arXiv:1202.2265 [math.QA].

\bibitem{Ramis}
J. P. Ramis: About the growth of entire functions solutions of linear algebraic  $q$-difference equations,
\emph{Ann.~Fac.~Sci.~Toulouse Math.}, {\bf 1}(6): 53-94 (1992).

\bibitem{Whittaker} J. M. Whittaker: On Lidstone's series and two-point expansion of analytic function, \emph{ Proc. London. Math. Soc.},
{\bf 36}(2): 451-469 (1934).

\end{thebibliography}
\end{document}